\documentclass[
  a4paper, 
  reqno, 
  oneside, 
  12pt
]{amsart}

\usepackage[utf8]{inputenc}

\usepackage[dvipsnames]{xcolor} % Must come before tikz!
\colorlet{cite}{red}
\usepackage{tikz}  % for diagrams

\usetikzlibrary{arrows,positioning}
\tikzset{ 
  baseline=-2.3pt,
  text height=1.5ex, text depth=0.25ex,
  >=stealth,
  node distance=2cm,
  mid/.style={fill=white,inner sep=2.5pt},
}

\usepackage{lmodern}
\usepackage{tikz-cd}
\usepackage{amsthm, amssymb, amsfonts}

\usepackage{amsthm, amssymb, amsfonts}	% Standard maths packages.
\usepackage[all]{xy}
\usepackage{graphicx,caption,subcaption}
\usepackage[font=small,labelfont=bf]{caption}
\usepackage{braket}  % For nice sets
\usepackage{microtype}
\usepackage{DotArrow}
\usepackage[makeroom]{cancel}
\usepackage[%
  bookmarks=true,			% show bookmarks bar?
  unicode=true,			% non-Latin characters in Acrobat’s bookmarks
  pdftitle={MonodromyProblem},		%
  pdfauthor={LopezGarcia-Valencia},	% 
  pdfkeywords={Monodromy problem, tangential-center problem},	% list of keywords
  colorlinks=true,		% false: boxed links; true: colored links
  linkcolor=black,			% color of internal links
  citecolor=black,		% color of links to bibliography
  filecolor=magenta,		% color of file links
  urlcolor=RoyalBlue			% color of external links
]{hyperref}				% One of the most useful packages I have found.
\usepackage{cleveref}

\newtheoremstyle{mydef}
  {}		% Space above environment
  {}		% Space below environment
  {}		% Body font
  {}		% Indent amount (empty = no indent, \parindent = para indent)
  {\scshape}	% theorem head font
  {. }		% Punctuation after heading
  { }		% Space after heading
  {\thmname{#1}\thmnumber{ #2}\thmnote{ #3}}	% Heading spec

\newtheorem{theorem}{Theorem}[section]
\newtheorem*{theorem*}{Theorem}

\newtheorem*{proposition*}{Proposition}
\newtheorem{lemma}[theorem]{Lemma}
\newtheorem*{lemma*}{Lemma}

\newtheorem*{corollary*}{Corollary}
\theoremstyle{definition}
\newtheorem{definition}[theorem]{Definition}
\newtheorem{example}[theorem]{Example}

\theoremstyle{remark}
\newtheorem{remark}[theorem]{Remark}

\newtheorem*{conjecture*}{Conjecture}
\usepackage{tikz}

\author{Daniel L\'opez-Garcia and Fabricio Valencia}
\address{}
\date{\today}

\address{D. L\'opez-Garcia, F. Valencia - Instituto de Matem\'atica e Estat\'istica, Universidade de S\~ao Paulo, Rua do Mat\~ao 1010, Cidade Universit\'aria, 05508-090 S\~ao Paulo - Brazil. \newline  
      \phantom{xx}
 dflopezga@ime.usp.br, fabricio.valencia@ime.usp.br}

\title{On the monodromy action for $f(x,y)=g(x)+h(y)$}

\begin{document}
\maketitle

\begin{abstract}
We solve the problem of determining under which conditions the monodromy of a vanishing cycle generates the whole homology of a regular fiber for a  polynomial $f(x,y)=g(x)+h(y)$ where $h$ and $g$ are polynomials with real coefficients having real critical points and satisfying $\deg(g)\cdot\deg(h)\leq 2500 $ and $\gcd(\deg(g),\deg(h))\leq 2$. As an application, under the same assumptions we solve the tangential-center problem which is known to be linked to the weak 16th Hilbert problem.  

\end{abstract}

%\tableofcontents
\section{Introduction}

Let $f(x,y)\in \mathbb{C}[x,y]$ be a tame polynomial with set of critical values $C$ and Milnor number $\mu<\infty$, see Definition \ref{TamePolynomial}. It follows that there are 1-dimensional \emph{vanishing cycles} $\delta_1,\delta_2, \ldots, \delta_{\mu}$ associated to the critical values of $f$, so that they span the $1$-homology  $H_{1}(f^{-1}(b))$ of a regular fiber $f^{-1}(b)$. Besides, there is an action $\textbf{Mon}$ of the fundamental group $\pi_1(\mathbb{C}\setminus C)$ on the $1$-homology $H_1(f^{-1}(b))$ which is given by local trivialization  of  $f^{-1}(\gamma)$ where $\gamma$ is any loop in $\mathbb{C}\setminus C$. We shall refer to such an action as the \emph{monodromy action} of $f$. Given a cycle $[\delta]\in H_1(f^{-1}(b))$ we get that the \emph{Picard-Lefschetz formula} computes the monodromy of $\delta$ around a critical value $c\in C$ by means of the expression
$$
	\textbf{Mon}_{c}(\delta)=\delta-\sum_k\langle \delta, \delta_k\rangle\delta_k,
$$
where $k$ runs through all the vanishing cycles that vanish at critical points inside $f^{-1}(c)$ and $\langle \delta, \delta_k\rangle$ stands for the intersection number between $\delta$ and $\delta_k$. The reader is recommended to visit \cite[c. 1-2]{AGV}, \cite[\S 6.3, \S 7.4]{M}, \cite{La} for specific details.

A natural question arises regarding whether the monodromy action on a generic fiber is transitive or not. Of course, we may rephrase such a question by asking under which conditions the monodromy of a vanishing cycle generates the whole $1$-homology of a regular fiber. This question turns out to have lots of interesting implications and is linked to several classical problems, as the weak version of the well known 16th Hilbert problem \cite{A} and the tangential version of the classical center–focus problem that goes back to Poincar\'e \cite{P} and Lyapunov \cite{Ly} and asks to characterize the centers of planar polynomial vector fields. Ilyashenko in \cite{I} introduced into the picture an approach to tackle the tangential center-focus problem by studying the monodromy in the case of generic polynomials. To understand the behavior of the monodromy action also brings important insights when studying polynomial foliations in $\mathbb{C}^2$. In particular, it has been used to describe some of the irreducible components of the algebraic set of polynomial foliations of certain fixed degree with at least one Morse center \cite{I,LiN,M2,Z}. To describe all the irreducible components of such an algebraic set is still an open problem.

The aim of this short paper is to solve the monodromy problem for polynomials $f(x,y)=g(x)+h(y)$ where $h$ and $g$ are polynomials with real coefficients having real critical points and satisfying $\deg(g)\cdot\deg(h)\leq 2500$ and $\gcd(\deg(g),\deg(h))\leq 2$, see Theorem \ref{MainThm1}. As an application, under the same assumptions we also solve the related tangential-center problem, see Theorem \ref{MainThm2}. These results clearly generalize, up to some degree restrictions,  the main results proved in \cite{CM} for hyperelliptic curves and the main results in \cite{LG1} for more general algebraic curves. Our approach relies in exploiting the combinatorial properties derived from the Dynkin diagram of $f$ as well as its intersection matrix of vanishing cycles. 

On the one hand, it is important to point out that the degree restriction over $g$ and $h$ depends on some software computations which were checked up to $\deg(g)\cdot\deg(h)\leq 2500$, see Lemma \ref{label: prop orbits for 1 critical value}. As expected, this condition could be actually enlarged to higher degrees, but it requires a lot of compilation time. On the other hand, the condition $\gcd(\deg(g),\deg(h))\leq 2$ in the checked cases implies that the eigenvalues of the intersection matrix are different, something that is needed to determine the Krylov subspaces associated to the vanishing cycles. Otherwise, such subspaces can not be determined by following our approach. For instance, the case $\deg(g)=\deg(h)=4$ was treated in \cite{LG1} to exemplify the additional algebraic conditions between $g$ and $h$ that characterize the subspaces generated by the orbit through a vanishing cycle with respect to the monodromy action. Based on these facts we predict that our main results are still valid for polynomials $f(x,y)=g(x)+h(y)$ where $h$ and $g$ are polynomials with real coefficients having real critical points and satisfying $\gcd(\deg(g),\deg(h))\leq 2$. 

\vspace{.2cm}
{\bf Acknowledgments:} We thank Hossein Movasati for helpful and insightful discussions. We are also grateful to Milo López for his useful suggestions on the computer calculations parts. L\'opez-Garcia was supported by Grant 2022/04705-8 Sao Paulo Research Foundation - FAPESP. Valencia was supported by Grant 2020/07704-7 Sao Paulo Research Foundation - FAPESP.

\section{The monodromy problem}

In this section we study the monodromy problem for polynomials of the form $f(x,y)=g(x)+h(y)$ where $h$ and $g$ are polynomials with real coefficients having real critical points. In order to do so, we need first to describe the intersection matrix of the vanishing cycles in a regular fiber $f^{-1}(b)$, which can be defined after looking at the join cycles construction and the intersection formula introduced in \cite[\S 7.10]{M} and widely explored in \cite{LG1}. Let us briefly introduce some necessary terminology. By considering a deformation of $f$ we may assume that it is a Morse function and that all of its critical values are pairwise distinct, see \cite{AGV}. Also, up to translations, we may further assume that the critical values of $g$ are positive, the critical values of $h$ are negative, and $b=0$ is a regular value for both polynomials. It is worth mentioning that the intersection matrix does not change after performing small deformations of $f$, although the monodromy action does. For instance, a polynomial can be deformed to a generic polynomial without changing the intersection matrix, but in the generic case the monodromy action is transitive, compare \cite{I}. Let $c^h_i$, $i=1,\cdots,e-1$ and $c^g_j$, $j=1,\cdots,d-1$ respectively denote the critical values of $h$ and $g$, and let $r_i$ and $s_j$ be paths from $0$ to $c_i^h$ and $c^g_j$ such that they have no self-intersections and intersect each other only at $0$. We fix the order $(s_1,s_2,\ldots, s_{d-1}, r_1,r_2,\ldots r_{e-1})$ about $0$, in such a way each path is anticlockwise oriented, as depicted in Figure \ref{setofpaths} below. The labeling of the paths has been chosen such that  $0< c^g_1< c^g_2< \cdots< c^g_{d-1}$ and  $c^h_{e-1}< c^h_{e-2}< \cdots < c^h_{1}< 0$.

\begin{figure}[h!]
	\centering
	\begin{footnotesize}
		\begin{tikzpicture}
			\matrix (m) [matrix of math nodes, row sep=0.4em,
			column sep=1.5em]{
				c^h_{e-1} &c^h_{e-2}&\ldots		&c^h_1&0&c^g_1&c^g_2&\ldots&	c^g_{d-1},\\
			};			
			\path[-stealth]
			(m-1-5) edge [bend left=45] node [above] {$r_{e-1}$}  (m-1-1) 
			edge [bend left=45] node [above] {$r_{e-2}$} (m-1-2)
			edge [bend left=45] node [above] {$r_{1}$}  (m-1-4) 
			
			edge [bend left=45] node [below] {$s_{1}$} (m-1-6)
			edge [bend left=45] node [below] {$s_{2}$} (m-1-7)
			edge [bend left=45] node [below] {$s_{d-1}$} (m-1-9);
		\end{tikzpicture}
	\end{footnotesize}
	\caption{\footnotesize A distinguished set of paths from $0$ to the critical values of $h$ and $g$ in $\mathbb{C}$}
	\label{setofpaths}
\end{figure}
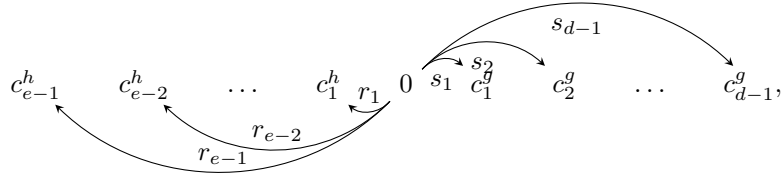

Consider now a critical value $c^g_j$ of $g$ with critical point $p_j$ and pick some $b'\in \mathbb{R}$, real value of $g$, near enough to $c^g_j$. As argued in \cite[\S 2]{LG1}, it follows that we may chose $b'$ such that the  real curve $x\mapsto (x,g(x))$ contains two points converging to $(p_j, c^g_j)$ as $b'$ goes to $c^g_j$. Let us denote them as $x_j$ and $x_{j+1}$. The 0-dimensional vanishing cycle $\sigma_j\in H_0(g^{-1}(b'))$ is defined as the formal sum $\sigma_j=x_{j+1}-x_{j}$. This cycle can be transported along a path from $b'$ to $0$ without monodromy to a cycle in $H_0(g^{-1}(0))$, which we keep denoting by $\sigma_j$. Analogously, 0-dimensional vanishing cycles $\gamma_{i}\in H_0(h^{-1}(0))$ can be defined by using instead the polynomial $h$. Besides, the formal sums above can be chosen so that the intersection between two 0-dimensional vanishing cycles is always $0$ or $-1$.
  
Let $\lambda=s_jr_i^{-1}$ be a path starting at $c^h_i$ and ending at $c^g_j$. The \emph{join cycle} $\gamma_i*\sigma_j$ is defined to be 
\begin{equation*}
	\gamma_i*\sigma_j:=\bigcup_{t\in[0,1]} \gamma_{i }(\lambda_t)\times \sigma_{j}( \lambda_t),
\end{equation*}  
where $\gamma_{i}(\lambda_t)\in  H_0(h^{-1}(\lambda_t))$ and $\sigma_{j }(\lambda_t)\in H_0(g^{-1}(\lambda_t))$. It holds that every join cycle $\gamma_i*\sigma_j$ is homeomorphic to a circle $S^1$, as Figure \ref{joinCycledegree2} below illustrates for the case $f(x,y)=x^2+y^2$. More importantly, they are 1-dimensional vanishing cycles and generate the whole homology $H_1(f^{-1}(0))$. The latter can be shown by using some technical but still classical results of Picard-Lefschetz theory, see \cite{La} and \cite[\S 7.5]{M}. If $\gamma_i*\sigma_j$ and $\gamma_{i'}*\sigma_{j'}$ are two join cycles along the paths  $\lambda:=s_jr_i^{-1}$ and $\lambda':=s_{j'}r_{i'}^{-1}$, respectively, then the intersection between them is determined by the formula
\begin{equation}
	\label{intersectionofjointcycles}
	\langle \gamma_i*\sigma_j, \gamma_{i'}*\sigma_{j'}\rangle=
	\left\{
	\begin{array}{lcl}
		\textnormal{sgn}(j'-j)\langle \sigma_j, \sigma_{j'}\rangle &\text{ if }\hspace{2mm}i=i'\text{ and }j\neq j'\\
		\textnormal{sgn}(i'-i)\langle \gamma_i, \gamma_{i'}\rangle &\text{ if }\hspace{2mm}j=j'\text{ and }i\neq i'\\
		\textnormal{sgn}(i'-i)\langle \gamma_i, \gamma_{i'}\rangle \langle \sigma_j, \sigma_{j'}\rangle &\hspace{3mm}\text{ if }\hspace{2mm} (i'-i)(j'-j)>0\\
		\hspace{2cm}0 &\hspace{4mm}\text{ if }\hspace{2mm} (i'-i)(j'-j)< 0.
	\end{array}
	\right.
\end{equation}

Due to our purposes, in what follows those join cycles are the only vanishing cycles we shall be considering. For specific details the reader is recommended to visit \cite[\S 7.10]{M} and \cite{LG1}.

\begin{center}
	\includegraphics[height=6cm]{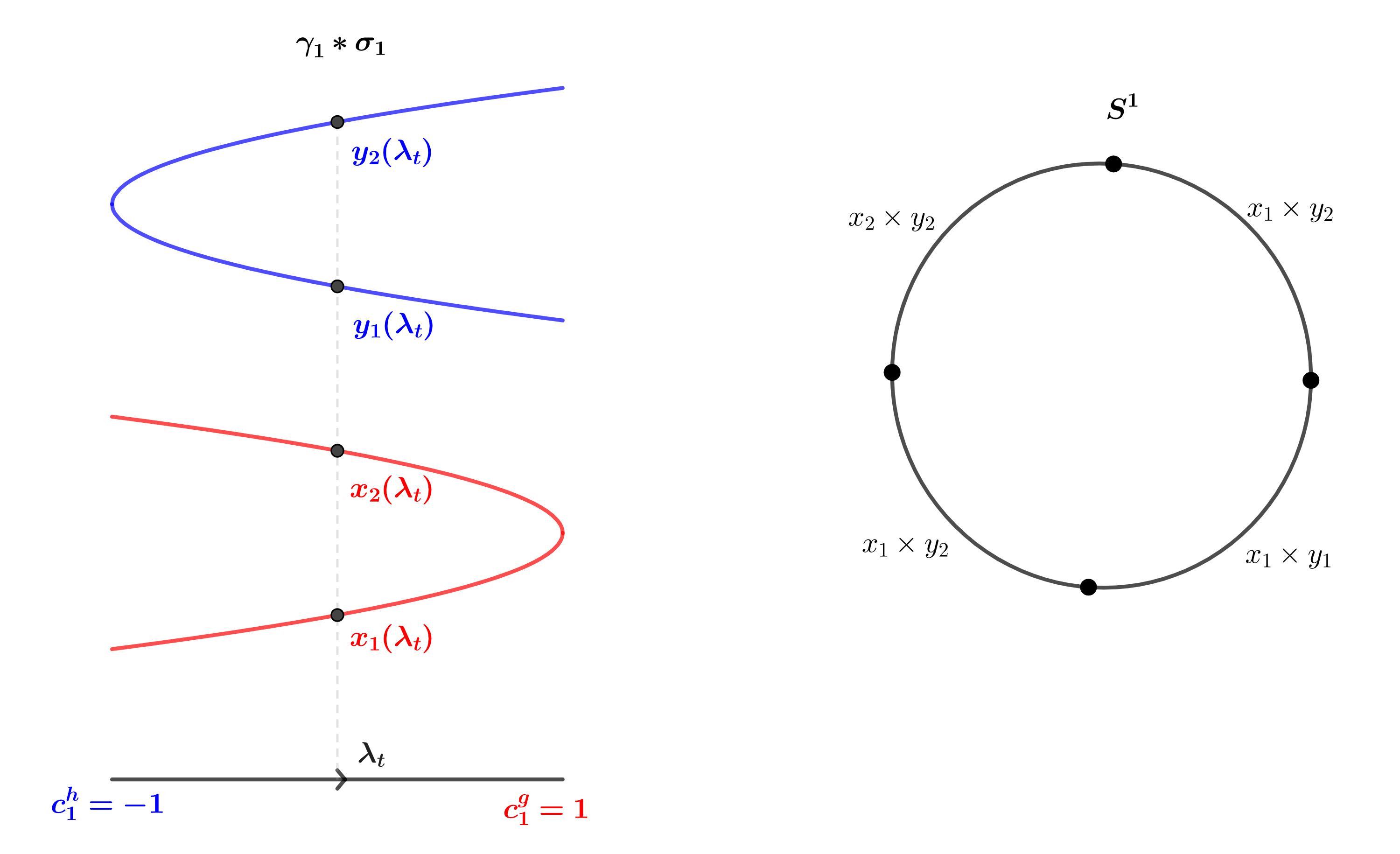}
	\captionof{figure}{\footnotesize Join cycle $\gamma_1\ast \sigma_1$ as $S^1$ for $f(x,y)=x^2+y^2$}
	\label{joinCycledegree2}
\end{center}

There is a combinatorial way of representing the intersection form, see \cite{AC, AGV}.

\begin{definition}
 The \emph{Dynkin diagram} of $f(x,y)=g(x)+h(y)$ is a 
 directed graph where the vertices correspond to the vanishing cycles in a regular fiber and the edges are determined by the possible intersections between the vanishing cycles. The vertices $\delta_i$ and $\delta_j$ are joined by an edge with multiplicity $\vert\langle \delta_i, \delta_j\rangle\vert$. If $\langle \delta_i, \delta_j\rangle>0$ then the edge points from $\delta_i$ to $\delta_j$.
\end{definition}

It is clear from the very definition that a Dynkin diagram is closely related to the intersection matrix $\Psi=(\langle \delta_i, \delta_j\rangle)$. Also, it is important to mention that we can also relate the vertices in a Dynkin diagram with the critical values associated to the vanishing cycles. In consequence, from the previous choice of paths $r_i$ and $s_j$ it follows that there are rules in the Dynkin diagram for $f(x,y)=g(x)+h(y)$ which establish the possibilities to relate the critical values. Firstly, the Dynkin diagram can be thought of as a two dimensional arrangement where the rows contain the critical values $c^h_i+c^g_j$ for a fixed $i$ and $j=1,\cdots , d-1$. Therefore, if two critical values of $f(x, y)$ in the same row of its Dynkin diagram agree then for the columns of these critical values there are equalities in the rows. Secondly, if $c^h_i+c^g_j=c^h_k+c^g_l$ with $i <k$ and  $j>l$ then $c^h_i=c^h_k$ and $c^g_j=c^g_l$. Similarly, the Dynkin diagram in dimension $0$ consists of vertices which correspond to the vanishing cycles associated to a polynomial and dashed edges representing an intersection of $-1$, as in this case the edges do not have direction. Visit \cite[\S 2]{LG1}.

Let us illustrate the previous notions with a concrete example.
\begin{example} 
	\label{ejemplo:g+h}
Set $g(x)=(x+3)(x+2)(x+1)(x-1)(x-2)(x-4)$ 
	and $h(y)=(3-y)(y-1)(y+1)(y+2)$. In dimension $0$ the Dynkin diagrams for $g$ and $h$ are respectively given by
	\begin{center}
		\begin{tikzpicture}           
			\matrix (m) [matrix of math nodes, row sep=0.5em,
			column sep=0.7em]{
				\sigma_3&\sigma_4 &\sigma_2 & \sigma_5&\sigma_1\\
			};			
			\path[-stealth]
			(m-1-1) edge [-, densely dashed]   (m-1-2)
			(m-1-2)	edge [-, densely dashed]	(m-1-3)
			(m-1-3)	edge [-, densely dashed]	(m-1-4)
			(m-1-4)	edge [-, densely dashed]	(m-1-5);
		\end{tikzpicture} and 
		\begin{tikzpicture}           
			\matrix (m) [matrix of math nodes, row sep=0.7em,
			column sep=0.7em]{
				\gamma_2&\gamma_3 &\gamma_1.\\
			};			
			\path[-stealth]
			(m-1-1) edge [-, densely dashed]   (m-1-2)
			(m-1-2)	edge [-, densely dashed]	(m-1-3)
			;
		\end{tikzpicture}
	\end{center}

The dashed lines above are classically used to indicate that the intersection number of two vanishing cycles equals $-1$. By using Formula \eqref{intersectionofjointcycles} we deduce that that Dynkin diagram for $f(x,y)=g(x)+h(y)$ can be depicted as
		\begin{footnotesize}
			\begin{equation}
			\label{Dynkindiagramgenerald}
			\begin{tikzpicture}           
			\matrix (m) [matrix of math nodes, row sep=1em,
			column sep=1em]{
				\gamma_2*\sigma_3&\gamma_2*\sigma_4&\gamma_2*\sigma_2&\gamma_2*\sigma_5&\gamma_2*\sigma_1\\
				\gamma_3*\sigma_3&\gamma_3*\sigma_4&\gamma_3*\sigma_2&\gamma_3*\sigma_5&\gamma_3*\sigma_1\\
				\gamma_1*\sigma_3&\gamma_1*\sigma_4&\gamma_1*\sigma_2&\gamma_1*\sigma_5&\gamma_1*\sigma_1.\\
				%c^h_2+c^g_3&c^h_2+c^g_4&c^h_2+c^g_2&c^h_2+c^g_5&c^h_2+c^g_1\\
				%c^h_1+c^g_3&c^h_1+c^g_4&c^h_1+c^g_2&c^h_1+c^g_5&c^h_1+c^g_1\\
				%c^h_3+c^g_3&c^h_3+c^g_4&c^h_3+c^g_2&c^h_3+c^g_5&c^h_3+c^g_1.\\
			};			
			\path[-stealth]
			(m-2-1) edge (m-1-1)
			(m-2-1)	edge (m-3-1)
			(m-2-2) edge (m-1-2)
			(m-2-2)	edge (m-3-2)
			(m-2-3) edge (m-1-3)
			(m-2-3)	edge (m-3-3)
			(m-2-4) edge (m-1-4)
			(m-2-4)	edge (m-3-4)
			(m-2-5) edge (m-1-5)
			(m-2-5)	edge (m-3-5)
			
			(m-1-2) edge (m-1-1)
			(m-1-2) edge (m-1-3)
			(m-1-4) edge (m-1-3)
			(m-1-4) edge (m-1-5)		
			(m-2-2) edge (m-2-1)
			(m-2-2) edge (m-2-3)
			(m-2-4) edge (m-2-3)
			(m-2-4) edge (m-2-5)
			(m-3-2) edge (m-3-1)
			(m-3-2) edge (m-3-3)
			(m-3-4) edge (m-3-3)
			(m-3-4) edge (m-3-5)

			(m-1-1) edge (m-2-2)
			(m-3-1) edge (m-2-2)
			(m-1-3) edge (m-2-2)
			(m-3-3) edge (m-2-2)
			(m-1-3) edge (m-2-4)
			(m-3-3) edge (m-2-4)
			(m-1-5) edge (m-2-4)
			(m-3-5) edge (m-2-4)		
			;
			\end{tikzpicture}
			\end{equation}
		\end{footnotesize}
We index the join cycles starting with the cycle at the top left of the Dynkin diagram and then continuing with the cycles in the same column from top to bottom. This indexing process is applied to the next column and so on. In this particular case such an indexing is given as $\{\delta_1,\ldots\delta_{15}\}=\{\gamma_2*\sigma_3,\gamma_3*\sigma_3,\gamma_1*\sigma_3,\gamma_2*\sigma_4,\cdots ,\gamma_2*\sigma_1,\gamma_3*\sigma_1,\gamma_1*\sigma_1\}$. Thus, the intersection matrix in the ordered basis $\delta_1,\ldots \delta_{12}$ is
\begin{footnotesize}
	$${\tiny \Psi=\left(\begin{array} {@{}*{15}{c}@{}}
		0& -1&  0& -1&  1&  0&  0&  0&  0&  0&  0&  0&  0&  0&  0\\
		1&  0&  1&  0& -1&  0&  0&  0&  0&  0&  0&  0&  0&  0&  0\\
		0& -1&  0&  0&  1& -1&  0&  0&  0&  0&  0&  0&  0&  0&  0\\
		1&  0&  0&  0& -1&  0&  1&  0&  0&  0&  0&  0&  0&  0&  0\\
		-1&  1& -1&  1&  0&  1& -1&  1& -1&  0&  0&  0&  0&  0&  0\\
		0&  0&  1&  0& -1&  0&  0&  0&  1&  0&  0&  0&  0&  0&  0\\
		0&  0&  0& -1&  1&  0&  0& -1&  0& -1&  1&  0&  0&  0&  0\\
		0&  0&  0&  0& -1&  0&  1&  0&  1&  0& -1&  0&  0&  0&  0\\
		0&  0&  0&  0&  1& -1&  0& -1&  0&  0&  1& -1&  0&  0&  0\\
		0&  0&  0&  0&  0&  0&  1&  0&  0&  0& -1&  0&  1&  0&  0\\
		0&  0&  0&  0&  0&  0& -1&  1& -1&  1&  0&  1& -1&  1& -1\\
		0&  0&  0&  0&  0&  0&  0&  0&  1&  0& -1&  0&  0&  0&  1\\
		0&  0&  0&  0&  0&  0&  0&  0&  0& -1&  1&  0&  0& -1&  0\\
		0&  0&  0&  0&  0&  0&  0&  0&  0&  0& -1&  0&  1&  0&  1\\
		0&  0&  0&  0&  0&  0&  0&  0&  0&  0&  1& -1&  0& -1&  0
		\end{array}\right).}$$
\end{footnotesize}
\end{example}

\begin{remark}
	The Dynkin diagram of $-f$ can be obtained by changing the direction of the edges in the Dynkin diagram $f$, so that its intersection matrix is given by minus the intersection matrix of $f$. The horizontal and vertical edges in the Dynkin diagram of $f$ can be determined by the 0-dimensional Dynkin diagrams of $g$ and $h$, with orientation according to Formula \eqref{intersectionofjointcycles}. Moreover, the diagonal edges in the Dynkin diagram are determined by the orientations of the horizontal and vertical edges, namely, each diagonal edge forms a well-oriented triangle with the other two edges. Therefore, given the integers numbers $d=\deg(g)$ and $e=\deg(h)$ there are only two possibilities for the Dynkin diagrams of $f$ up to a sign in the intersection matrix. Indeed, they are $f=g+h$ and $f=g-h$. 
\end{remark}

Let us now focus on studying the monodromy problem for polynomials $f(x,y) = g(x)+h(y)$. We want to compute the subspaces of $H_1(f^{-1}(0))$ generated by the orbits of the join cycles under the monodromy action. This shall allow us to characterize the cases where those subspaces are not the whole space $H_1(f^{-1}(0))$. We start by considering the case where we only have one critical value, so that we suppose $f_0(x,y)=x^d+y^e$. By the Picard-Lefschetz formula it follows that to study the monodromy problem associated to $f_0$ is equivalent to studying the so-called \emph{Krylov subspaces} $\text{span}\{v,\Psi v,\Psi^2 v,\cdots, \Psi^{N-1}v\}$ for the intersection matrix $\Psi$ and each vanishing cycle $v$, where $N=(d-1)(e-1)$. These subspaces can be computed numerically for several values of $d$ and $e$. Nevertheless, in order to avoid numerical errors when computing powers of the matrix $\Psi$ we get the Krylov subspaces by noting the following key fact. Observe that the matrix $\Psi$ is diagonalizable (it is skew-symmetric), so we can write $v=\sum_j r_j u_j$ where the $u_j's$ are the eigenvectors of $\Psi$. Let $\lambda_j$ denote the eigenvalue of $u_j$ and set $\Psi^lv=\sum_{j=1}^{N} r_j \lambda_j^l u_j$.
Hence, 
 $$\begin{pmatrix}
	v& \Psi v& \Psi^2v&\cdots& \Psi^{N-1}v
\end{pmatrix}=\begin{pmatrix}
	r_1 u_1&r_2u_2&\cdots&r_Nu_n
\end{pmatrix}
\begin{pmatrix}
	1&\lambda_1&\lambda_1^2&\ldots&\lambda_1^{N-1}\\
	1&\lambda_2&\lambda_2^2&\ldots&\lambda_2^{N-1}\\
	\vdots&\vdots&\vdots&&\vdots\\
	1&\lambda_N&\lambda_N^2&\ldots&\lambda_N^{N-1}\\
\end{pmatrix}.$$

Note that the determinant of the Vandermonde matrix in the right hand side equals $\prod_{i<j}(\lambda_j-\lambda_i)$. Thus, if the eigenvalues are different, then the Krylov subspace of $v$ is the span of the vectors $u_l$ for which $r_l\neq 0$. In consequence, we can determine these subspaces by computing the eigenvalues and eigenvectors of $\Psi$ instead of computing  powers of $\Psi$, compare \cite[\S 3]{LG1}.
\begin{remark}
The previous notations give rise to a simpler indexing for the 1-dimensional vanishing cycles in the Dynkin diagram of $f$, namely, the one given by using the standard matrix indexing. For instance, the Dynkin diagram \eqref{Dynkindiagramgenerald} can be also depicted as
\begin{footnotesize}
	\begin{equation}
	\label{Dynkindiagramgenerald2}
	\begin{tikzpicture}           
	\matrix (m) [matrix of math nodes, row sep=1em,
	column sep=1em]{
		v_{1,1}&v_{1,2}&v_{1,3}&v_{1,4}&v_{1,5}\\
		v_{2,1}&v_{2,2}&v_{2,3}&v_{2,4}&v_{2,5}\\
		v_{3,1}&v_{3,2}&v_{3,3}&v_{3,4}&v_{3,5}\\
	};			
	\path[-stealth]
	(m-2-1) edge (m-1-1)
	(m-2-1)	edge (m-3-1)
	(m-2-2) edge (m-1-2)
	(m-2-2)	edge (m-3-2)
	(m-2-3) edge (m-1-3)
	(m-2-3)	edge (m-3-3)
	(m-2-4) edge (m-1-4)
	(m-2-4)	edge (m-3-4)
	(m-2-5) edge (m-1-5)
	(m-2-5)	edge (m-3-5)
	
	(m-1-2) edge (m-1-1)
	(m-1-2) edge (m-1-3)
	(m-1-4) edge (m-1-3)
	(m-1-4) edge (m-1-5)		
	(m-2-2) edge (m-2-1)
	(m-2-2) edge (m-2-3)
	(m-2-4) edge (m-2-3)
	(m-2-4) edge (m-2-5)
	(m-3-2) edge (m-3-1)
	(m-3-2) edge (m-3-3)
	(m-3-4) edge (m-3-3)
	(m-3-4) edge (m-3-5)

	(m-1-1) edge (m-2-2)
	(m-3-1) edge (m-2-2)
	(m-1-3) edge (m-2-2)
	(m-3-3) edge (m-2-2)
	(m-1-3) edge (m-2-4)
	(m-3-3) edge (m-2-4)
	(m-1-5) edge (m-2-4)
	(m-3-5) edge (m-2-4)		
	;
	\end{tikzpicture} \quad \leftrightsquigarrow \quad \begin{tikzpicture}           
	\matrix (m) [matrix of math nodes, row sep=1em,
	column sep=1em]{
		c_{1,1}&c_{1,2}&c_{1,3}&c_{1,4}&c_{1,5}\\
		c_{2,1}&c_{2,2}&c_{2,3}&c_{2,4}&c_{2,5}\\
		c_{3,1}&c_{3,2}&c_{3,3}&c_{3,4}&c_{3,5},\\
	};			
	\path[-stealth]
	(m-2-1) edge (m-1-1)
	(m-2-1)	edge (m-3-1)
	(m-2-2) edge (m-1-2)
	(m-2-2)	edge (m-3-2)
	(m-2-3) edge (m-1-3)
	(m-2-3)	edge (m-3-3)
	(m-2-4) edge (m-1-4)
	(m-2-4)	edge (m-3-4)
	(m-2-5) edge (m-1-5)
	(m-2-5)	edge (m-3-5)
	
	(m-1-2) edge (m-1-1)
	(m-1-2) edge (m-1-3)
	(m-1-4) edge (m-1-3)
	(m-1-4) edge (m-1-5)		
	(m-2-2) edge (m-2-1)
	(m-2-2) edge (m-2-3)
	(m-2-4) edge (m-2-3)
	(m-2-4) edge (m-2-5)
	(m-3-2) edge (m-3-1)
	(m-3-2) edge (m-3-3)
	(m-3-4) edge (m-3-3)
	(m-3-4) edge (m-3-5)

	(m-1-1) edge (m-2-2)
	(m-3-1) edge (m-2-2)
	(m-1-3) edge (m-2-2)
	(m-3-3) edge (m-2-2)
	(m-1-3) edge (m-2-4)
	(m-3-3) edge (m-2-4)
	(m-1-5) edge (m-2-4)
	(m-3-5) edge (m-2-4)		
	;
	\end{tikzpicture}
	\end{equation}
\end{footnotesize}
where $v_{i,j}$ is the vanishing cycle corresponding to the critical value $c_{i,j}$ which, in turn, corresponds to the position $(i,j)$ in the Dynkin diagram \eqref{Dynkindiagramgenerald}. In general, these two diagrams express the relation between $\pi_1(\mathbb{C}\setminus C)$ which is generated by loops around the critical values of $f$ and some of the automorphisms of $H_1(f^{-1}(0))$ associated to the intersection matrices. This can be realized once we compute the monodromy action by means of the Picard-Lefschetz formula.

\end{remark}

\begin{lemma}
	\label{label: prop orbits for 1 critical value}
Let $f_0(x,y)=x^d+y^e$ be a polynomial with $d\cdot e\leq 2500 $ and $\gcd(d,e)\leq 2$. Let $v_{i,j}$ be a vanishing cycle at the Dynkin diagram of $f_0(x,y)$. Then, the span of the orbit by the monodromy action on $v_{i,j}$ contains the following linear combinations:
	
	$$v_{i,mp}, \qquad v_{i,mp-k}+v_{i,mp+k},\qquad v_{i-1,mp-k}+v_{i-1,mp+k}+v_{i+1,mp-k}+v_{i+1,mp+k},$$  $$v_{nr,j},\qquad v_{nr-l,j}+v_{nr+l,j},\qquad v_{nr-l,j-1}+v_{nr+l,j-1}+v_{nr-l,j+1}+v_{nr+l,j+1},$$

where $p=\gcd(d,j)$ and $r=\gcd(e,i)$, $m=1,\ldots,  \frac{d}{p}-1$, $k=1,\ldots, p-1$, $n=1,\ldots,  \frac{e}{r}-1$, $l=1,\ldots, r-1$. 
\end{lemma}
\begin{proof}
This result can be proven by using a computing program. The reader can use the functions written in MATLAB, \textit{VanCycleSub} and \textit{Proof\_lemma\_2\_5}, as a numerical supplement of this proof \footnote{https://github.com/danfelmath01/Mondromy-and-tangential-problems-for-direct-sum-of-polynomials}.

\end{proof}

\begin{example}
The span of the orbit by monodromy action on the vanishing cycle $v_{2,2}$ at the Dynkin diagram  \eqref{Dynkindiagramgenerald2} contains linear combinations of the form $v_{2,2},v_{2,4},v_{2,1}+v_{2,3}, v_{2,3}+v_{2,5}, v_{1,2}+v_{3,2}, v_{1,1}+v_{1,3}+v_{3,1}+v_{3,3}$
\end{example}

From now on we assume that $f(x,y)=g(x)+h(y)$ additionally satisfies $d=\deg(g)$ and $e=\deg(h)$ with $d\cdot e\leq 2500$ and $\gcd(d,e)\leq 2$. After doing small deformations in the polynomials $f$ and $f_0=x^d+y^e$ we may conclude that their Dynkin diagrams are equivalent. Let $G_f$ and $G_0$ denote the monodromy groups associated to $f$ and $f_0$, respectively. By considering paths surrounding all critical values of $f$ it follows that for any vanishing cycle $v$ the orbits by these actions satisfy $G_0\cdot v\subset G_f\cdot v$. Therefore, our approach to study the subspaces of $H_1(f^{-1}(0))$, which are generated by the orbits of the monodromy action, strongly depends on using Lemma \ref{label: prop orbits for 1 critical value}.

Motivated by the notions studied in \cite{LG1} we introduce a notion that relates some behaviors of the orbits by the monodromy action with some combinatorial aspects derived form the Dynkin diagrams.

\begin{definition}
	The Dynkin diagram of $f(x,y)=g(x)+h(y)$ is said to have
	\emph{horizontal symmetry} if there exists an integer $p >1$ such that for any $j$ with $\text{gcd}(j,d)=p$
	the critical values of $f$ satisfy
	$$c_{i,j-k}= c_{i,j+k}\quad\text{where } k = 1,\ldots,p-1, \text{ and }i=1,\ldots, e-1.$$

 Accordingly, the vanishing cycles $v_{i,lr}$ with $l = 1,\cdots, \frac{d}{r}-1$ are called \emph{vanishing cycles with horizontal symmetry}. A notion of \emph{vertical symmetry} can be analogously defined.
\end{definition}

\begin{lemma} \label{label: prop horizontal symmetries vs orbits}
	Suppose that the Dynkin diagram associated to $f$ has no horizontal symmetry. Then, the subspace generated by the orbit through $v_{i,j}$ with respect to the monodromy action contains all the vanishing cycles $v_{i,k}$ for $k=1,\ldots, d-1$.
\end{lemma}

\begin{proof}
	As consequence of Lemma \ref{label: prop orbits for 1 critical value} we get that the span of the orbit through $v_{i,j}$ contains the linear combinations $v_{i,mp}$, $v_{i,mp-k}+v_{i,mp+k}$ and $v_{i-1,mp-k}+v_{i-1,mp+k}+v_{i+1,mp-k}+v_{i+1,mp+k}$ where $p=\gcd(j,d)$. Since there is no horizontal symmetry there exists $j_0=m_0p$ such that $c_{i,j_0-k_0}\neq c_{i,j_0+k_0}$ for some $k_0\in \lbrace 1,\cdots, p-1\rbrace$. By considering the monodromy action on $v_{i,j_0-k_0-1}+v_{i,j_0+k_0-1}$ and $v_{i-1,j_0-k_0-1}+v_{i-1,j_0+k_0-1}+v_{i+1,j_0-k_0-1}+v_{i+1,j_0+k_0-1}$ around $c_{i,j_0+k_0}$ we obtain $v_{i,j_0+k_0}-v_{i-1,j_0+k_0}-v_{i+1,j_0+k_0}$ and $v_{i-1,j_0+k_0}+v_{i+1,j_0+k_0}$; $v_{i,j_0+k_0}-v_{i-1,j_0+k_0}-v_{i+1,j_0-k_0}$ and $v_{i-1,j_0+k_0}+v_{i+1,j_0-k_0}$; or $v_{i,j_0+k_0}$ directly.
	
	 In either case, we can produce the vanishing cycle $v_{i,j_0+k_0}$ and therefore $v_{i,j_0-k_0}$. If $\gcd(j_0+k_0,d)=1$ or $\gcd(j_0-k_0,d)=1$ then the results follows from Lemma \ref{label: prop orbits for 1 critical value}. Thus, by arguing as above inductively we obtain either a cycle $v_{i,j'}$ such that $\gcd(j',d)=1$ or else we produce vanishing cycles $v_{i,p_1}$ and $v_{i,p_2}$ where $p_1, p_2$ are the smallest numbers verifying $\gcd(p_i,d)\neq 1$ and $p_1<p_2<2p_1$. In the latter case, it is simple to check that these cycles produce the cycle $v_{i,2p_1-p_2}$.
\end{proof}
We are now in condition to state our first main result which clearly generalize, up to some degree restrictions,  the second main result proved in \cite{CM} for hyperelliptic curves and the first main result in \cite{LG1} for more general algebraic curves.

\begin{theorem}
	\label{MainThm1}
	Let $f(x,y)=g(x)+h(y)$ be a polynomial where $h$ and $g$ are polynomials with real coefficients having real critical points. Suppose that $d=\deg(g)$ and $e=\deg(h)$ with $d\cdot e\leq 2500$ and $\gcd(d,e)\leq 2$. If $\delta_t\subset f^{-1}(t)$ is a vanishing cycle then one of the following assertions holds true:
	\begin{enumerate}
		\item[1.] the orbit of $\delta_t$ by the monodromy action generates the homology $H_1(f^{-1}(t))$, or
		
		\item[2.] either the polynomial $g$ is decomposable as $g = g_2\circ g_1$ and $\pi_*(\delta_t)$ is homotopic to
		zero in $\{g_2(z)+h(y)= t\}$ where $\pi(x,y)=(g_1(x),y)=(z,y)$, or else the polynomial $h$ is decomposable as $h = h_2\circ h_1$ and $\pi_*(\delta_t)$ is homotopic to
		zero in $\{g(x)+h_2(z)= t\}$, where $\pi(x,y)=(x,h_1(y))=(x,z)$.
	\end{enumerate}
\end{theorem}
\begin{proof}
 
We start by claiming that if the orbit through $\delta_t$ with respect to the monodromy action does not generate the homology $H_1(f^{-1}(t))$ then the Dynkin diagram of $f$ has horizontal or vertical symmetry. Indeed, let $\delta_t$ be denoted by $v_{i,j}$ and let $V\subset H_1(f^{-1}(t))$ be the subspace generated by the monodromy orbit through $v_{i,j}$. From Lemma \ref{label: prop horizontal symmetries vs orbits} we know that if the Dynkin diagram of $f$ has no horizontal symmetry then the vanishing cycles in the same row containing $v_{i,j}$ belongs to $V$. If, in addition, such a diagram has no vertical symmetry then we get that each cycle in the row containing $v_{i,j}$ produces all the cycles belonging to its respective column. Therefore, if the Dynkin diagram of $f$ has neither horizontal symmetry nor vertical symmetry then we deduce that $V=H_1(f^{-1}(t))$. This fact implies that the vanishing cycles $\delta_{t}$ of $f$ for which their orbit by the monodromy action generate subspaces $W$ with $W\neq H_1(f^{-1}(t))$ are precisely those vanishing cycles having horizontal symmetry or vertical symmetry, provided the Dynkin diagram of $f$ has horizontal symmetry or vertical symmetry, respectively. 

Let us suppose that the orbit through $\delta_t$ with respect to the monodromy action does not generate the homology $H_1(f^{-1}(t))$. That is, $\delta_t$ is either a vanishing cycle with horizontal or vertical symmetry. Assume that $\delta_t$ has horizontal symmetry. It follows that whether the Dynkin diagram of $f$ has horizontal symmetry depends only on $g$, so that $\tilde f:=y^2+g(x)$ has horizontal symmetry as well. Furthermore, the cycle $\delta_t$ is in correspondence with a vanishing cycle $\tilde \delta_t$ of $\tilde f$ having horizontal symmetry. And so it is simple to see that the span of the orbit through  $\tilde \delta_t$ does not generate $H_1(\tilde f^{-1}(t))$. Therefore, by using Theorem 1.6 in \cite{CM} it follows that $g$ is decomposable as $g=g_2\circ g_1$. Moreover, we get that $\pi_*(\tilde \delta_t)$ is null-homotopic in $\{y^2+g_2(z)=t\}$ where $\pi(x,y)=(g_1(x),y)=(z,y)$. Note that from the definition of join cycle we deduce that $\pi_*(\delta_t)$ is also null-homotopic in $\{h(y)+g_2(z)=t\}$.

It is clear that if $\delta_t$ has vertical symmetry then we obtain that a similar decomposition for $h$ satisfying the null-homotopic property above turns out to be equivalent to saying that $f$ has vertical symmetry. So, the result follow as desired.
\end{proof}

\begin{remark}
The proof of Theorem 1.6 in \cite{CM} mainly uses L\"uroth's theorem for polynomials in one variable, which is closely related to the case  $g(x)+y^2$. When $\deg(h(y))>2$ there is no analogous result to L\"uroth's theorem we may use to bring a similar proof. However, as it was evidenced above, the combinatorial and linear results provided in Lemmas \ref{label: prop orbits for 1 critical value} and \ref{label: prop horizontal symmetries vs orbits} allow us to prove a result which generalizes the mentioned theorem in \cite{CM}.
\end{remark}

\section{The tangential-center problem}
In this section we provide an application of Theorem \ref{MainThm1} which is related to the \emph{weak 16th Hilbert problem}. Such a problem asks to bound the number of real limit cycles in the system $df+\epsilon \omega=0$ for small $\epsilon$, where $f$ is a polynomial and $\omega$ is a polynomial $1$-form on $\mathbb{C}^2$, visit \cite{A,SY}. This is closely linked to the study of the zeros of the Abelian integral $I(t)=\int_{\delta_t}\omega$ where $\delta_t$ is a 1-parameter family of cycles lying inside the fiber $f^{-1}(t)$. It follows that the limit cycles of the system $df+\epsilon \omega=0$ are  related to the zeros of $I(t)$ for generic values of $t$, provided this integral does not vanish identically. Thus, the weak 16th Hilbert problem may be rephrased in tangential terms by asking to bound the number of zeros of $I(t)$ in terms of the algebraic degrees of $f$ and $\omega$. Note that if the Abelian integral above vanishes identically then no information about the limit cycles of the system involved can be gained. Nevertheless, such a vanishing condition is related to the tangential version of the classical center–focus problem which asks to characterize the centers of planar polynomial vector fields \cite{Ly,P}, so that it is also of independent interest. For instance, Ilyashenko in \cite{I} characterized the conditions under which $I(t)$ vanishes identically when $f$ is generic. It was done by proving that the monodromy acts transitively on the first homology group of the generic fiber. Hence, in this case the vanishing of $I(t)$ implies the vanishing of the Abelian integral along all cycles in $H_1(f^{-1}(t))$, thus obtaining that $\omega$ is relatively exact. That is, it satisfies the so-called \emph{Condition $(*)$} which allows to give formulas for the first nonzero derivative of a period function, see \cite{F}.

Our aim below is to describe the implications of having that $I(t)$ vanishes identically for the cases $f(x,y)=g(x)+h(y)$ where $h$ and $g$ are polynomials with real coefficients having real critical points and satisfying $\deg(g)\cdot\deg(h)\leq 2500 $ and $\gcd(\deg(g),\deg(h))\leq 2$. In order to do so we need to take care first of the following technical results.

\begin{lemma}\label{Corollary1}
	Suppose that the polynomial $g$ is decomposable as $g = g_2\circ g_1$ and $\pi_*(\delta_t)$ is homotopic to
	zero in $\{g_2(z)+h(y)= t\}$ where $\pi(x,y)=(g_1(x),y)=(z,y)$. Assume that $\deg(g_1)=a$ and $\delta_t$ is at the position $\frac{dm}{a}$ with $m\in\{1,\ldots, a-1\}$. Set $F(z,y)=g_2(z)+h(y)$. Then, the induced group homomorphism $\pi_\ast: H_1(f^{-1}(t))\to H_1(F^{-1}(t))$ is surjective and  its kernel is generated by the orbit through $\delta_t$ with respect to the monodromy action.
\end{lemma}
\begin{proof}
	From the very definition of the join cycles we get that the result shall follow by analyzing the 0-dimensional Dynkin diagram. So, let us consider a vanishing for $g_2(z)=r$ associated to a critical point $p$. This is given by a formal sum $(z_2,t)-(z_1,t)$ such that $g_2(z_2)=g_2(z_1)=t$. We may further assume that the points $z_1$ and $z_2$ are the closest ones to $p$. Let $g_1^{-1}(z_1)=\{x_1^1, x_1^2,\ldots, x_1^a\}$ and $g_1^{-1}(z_2)=\{x_2^1, x_2^2,\ldots, x_2^a\}$. Firstly, the 0-cycles for $g_2\circ g_1(x)=t$ which are defined by $v=(x_2^i,t)-(x_1^j,t)$ satisfy that $\pi_*(v)=(z_2,t)-(z_1,t)$. That is, $\pi_\ast$ is surjective. Secondly, we know that the kernel of $\pi_*$ is generated by the 0-cycles of the form $(x_1^i,t)-(x_1^j,t)$ and $(x_2^i,t)-(x_2^j,t)$. For points $x_1^i$ and $x_1^j$ (resp. $x_2^i$ and $x_2^j$) which are adjacent we obtain that the difference $(x_1^i,t)-(x_1^j,t)$ is a vanishing cycle of $g_2\circ g_1(x)=t$ at the position $\frac{kd}{a}$. The differences $(x_2^i,t)-(x_2^j,t)$ of points  $x_2^i$ and $x_2^j$ which are not adjacent can be rewritten as $((x_2^i,t)-(x_1^k,t))-((x_2^j,t)-(x_1^l,t))+((x_1^k,t)-(x_1^l,t))$, where each of the differences defines a vanishing cycle. Hence, from Lemma \ref{label: prop orbits for 1 critical value} we obtain that the cycles  $((x_2^i,t)-(x_1^k,t))-((x_2^j,t)-(x_1^l,t))$ and $((x_1^k,t)-(x_1^l,t))$ are in the span of the orbit through $\delta_t$.
\end{proof}

The following notion can be found in \cite[\S 7.3]{M}.

\begin{definition}\label{TamePolynomial}
Let $f$ be a polynomial in $\mathbb{C}[x,y]$ and $q$ be its highest degree homogeneous piece. We say that $f$ is \emph{tame} if the $\mathbb{C}$-module $V_q:=\frac{\mathbb{C}[x,y]}{\textnormal{jacob}(q)}$ is finitely generated. In this case, we refer to the rank of $V_q$ as the Milnor number of $f$.
\end{definition}

If $f$ is tame then the \emph{Petrov module} $H_f=\frac{\Omega^1}{d\Omega^0+df\wedge \Omega^0}$, where $\Omega^i$ stands for the polynomials forms, is finitely generated. Moreover, each $\omega \in H_f$ can be written as $\omega=\sum_{i=1}^\mu p_i(f)\omega_i$ where the $p_i's$ are polynomials and the forms $\omega_i's$ are defined by the condition $d\omega_i=g_i dx \wedge dy$ with $g_1,\ldots, g_{\mu}$ being a monomial basis of $\frac{\mathbb{C}[x,y]}{\langle f_x, f_y, \rangle}$, compare \cite{G}.

\begin{lemma}\label{KeyComputation}
Let $\pi:\mathbb{C}^2\to \mathbb{C}^2$ be a polynomial function given by $(x,y) \to (u(x),v(y))$ and let $F\in \mathbb{C}[x,y]$ be a polynomial such that $f=F\circ \pi:\mathbb{C}^2\to \mathbb{C}$ is tame. Suppose that the induced homomorphism $\pi_\ast: H_1(f^{-1}(t))\to H_1(F^{-1}(t))$ is surjective for every regular value $t\in \mathbb{C}$ of $f$. If $\omega$ is a polynomial $1$-form on $\mathbb{C}^2$ for which $\int_\delta \omega=0$ for all $\delta\in \ker(\pi_\ast)$ then there is another polynomial $1$-form $\alpha$ on $\mathbb{C}^2$ such that $\int_\gamma(\pi^\ast(\alpha)-\omega)=0$ for all $\gamma\in H_1(f^{-1}(t))$.
\end{lemma}

\begin{proof}
	By using the de Rham isomorphism it is possible to define  $\alpha_t\in H_{dR}^1(F^{-1}(t))$ by means of the equality $\int_{\sigma_t} \alpha_t:=\int_{\gamma_t} \omega|_t$, where $\sigma_t \in H_1(F^{-1}(t))$ is such that $\pi_*(\gamma_t)=\sigma_t$. The form $\alpha_t$ turns out to be well defined since $\pi_*$ is surjective and $\int_\delta \omega=0$ for all $\delta\in \ker(\pi_\ast)$ by hypothesis. These differential $1$-forms $\alpha_t$ can be taken algebraically, see \cite[p. 75]{MV}. Moreover, they yield a holomorphic global section $\alpha$ outside the critical values of $f$ such that $\alpha$ restricted to $F^{-1}(t)$ equals $\alpha_t$ for all regular value $t\in \mathbb{C}$. It remains to prove that $\alpha$ is a polynomial $1$-form on  $\mathbb{C}^2$. Each of the holomorphic forms on $F^{-1}(t)$  can be written as $\phi\cdot\textnormal{res}\left[\frac{dx\wedge dy}{F-t}\right]$ where $\phi$ is holomorphic on $F^{-1}(t)$. Therefore, $\alpha=\eta(t) \phi \cdot \textnormal{res}\left[\frac{dx\wedge dy}{F-t}\right]$ with $\eta$ being holomorphic outside the critical values of $f$. Let us consider a family of cycles $\sigma_t\in H_1(F^{-1}(t))$ such that $\int_{\sigma_t}\phi \cdot \textnormal{res}\left[\frac{dx\wedge dy}{F-t}\right]=1$, so that $\int_{\sigma_t} \alpha=\eta(t)$. There exists a polynomial $s(t)$ such that $\lim_{t\to c}s(t)\int_{\gamma_t}\omega<\infty$ for each critical value $c$ of $f$, see \cite[c. 10.1.3]{AGV}. This implies that $\lim_{t \to c} s(t) \eta(t)=\lim_{t\to c}s(t)\int_{\gamma_t}\omega<\infty$ . By the Riemann extension theorem we get that $s\eta$ is holomorphic on $\mathbb{C}$ and therefore $s\alpha$ is holomorphic on $\mathbb{C}^2$. Furthermore, since $\omega$ is a polynomial $1$-form, by the property about Petrov modules mentioned above we obtain that $\omega=\sum_{i=1}^\mu p_i(t)\omega_i+dA+Bdf$, meaning that there exists a natural number $N$ such that $$\lim_{t\to \infty}\frac{s(t)\eta(t)}{t^N}=\lim_{t\to \infty}\frac{s(t)\sum_{i=1}^\mu p_i(t)\int_{\gamma_t}\omega_i}{t^N}<\infty.$$
	
	Consequently, $s\eta$ is a polynomial which means that $s\pi^\ast(\alpha)$ is also a polynomial $1$-form. Thus, there are polynomials $q_i$ such that $s \pi^\ast (\alpha)=\sum_{i=1}^\mu q_i(f)\omega_i$. But $s\pi^\ast \alpha=s\omega$ in $H_f$, so that  $\sum _{i=1}^\mu q_i(f)\omega_i=\sum_{i=1}^\mu s(f) p_i(f)\omega_i$. Therefore, $s$ divides the $q_i$'s from which we deduce that $\pi^\ast\alpha$ is a polynomial $1$-form. Finally, if $\alpha=\alpha_1(x,y)dx+\alpha_2(x,y)dy$ then it is simple to check by arguing similarly that $\alpha_1$ and $\alpha_2$ are polynomials since both $s\alpha$ and $\pi^\ast\alpha=\alpha_1(u(x),v(y))u'(x)dx+\alpha_2(u(x),v(y))v'(y)dy$ are polynomial 1-forms.
\end{proof}

We can now state a nice application of Theorem \ref{MainThm1} which provides a generalization of the first main theorem in \cite{CM}, up to the same degree restrictions we have been working throughout. Namely:

\begin{theorem}
	\label{MainThm2}
	Let $f(x,y)=g(x)+h(y)$ be a polynomial where $h$ and $g$ are polynomials with real coefficients having real critical points. Suppose that $d=\deg(g)$ and $e=\deg(h)$ with $d\cdot e\leq 2500$ and $\gcd(d,e)\leq 2$. If $\omega$ is a polynomial 1-form on $\mathbb{C}^2$ and $\delta_t\subset f^{-1}(t)$ is a vanishing cycle then $\int_{\delta_t}\omega=0$ if and only if one of the following assertions holds true:
	\begin{enumerate}
		\item[1.]  $\omega$ is relatively exact, meaning that it can be written as $\omega=Adf+dB$, or
		\item[2.]  $\omega$ can be written as $\omega=\pi^*\alpha+Adf+dB$ where $\alpha$ is a polynomial $1$-form and $\pi(x,y)=(g_1(x),y)$ or $\pi(x,y)=(x,h_1(y))$ are maps satisfying the conditions from item 2 in Theorem \ref{MainThm1}.
	\end{enumerate}
Here $A,B\in\mathbb{C}[x,y]$.
\end{theorem}

\begin{proof}
Let us take $\delta_t\subset f^{-1}(t)$ a vanishing cycle such that $\int_{\delta_t}\omega=0$. It follows from Theorem \ref{MainThm1} that if the orbit through $\delta_t$ with respect to the monodromy action generates $H_1(f^{-1}(t))$ then $\int_{\delta}\omega=0$ for all $\delta\in H_1(f^{-1}(t))$ which is equivalent to $\omega$ being relatively exact. Therefore, there are polynomials $A,B\in\mathbb{C}[x,y]$ such that $\omega=Adf+dB$, visit \cite{G}. Otherwise, let us suppose, without loss of generality, that the polynomial $g$ is decomposable as $g = g_2\circ g_1$ and $\pi_*(\delta_t)$ is homotopic to
zero in $\{g_2(z)+h(y)= t\}$ where $\pi(x,y)=(g_1(x),y)=(z,y)$. Set $F(z,y)=g_2(z)+h(y)$. From Lemma \ref{Corollary1} we know that the induced homomorphism $\pi_\ast: H_1(f^{-1}(t))\to H_1(F^{-1}(t))$ is surjective and that its kernel is generated by the orbit through $\delta_t$ with respect to the monodromy action. Hence, the result follows by applying Lemma \ref{KeyComputation} and by arguing as we did above with the difference $\omega-\pi^\ast(\alpha)$ instead. The converse is straightforward since $\delta_t$ is a vanishing cycle inside the fiber $f^{-1}(t)$.
\end{proof}

Based on the results obtained in this paper we predict that:

\begin{conjecture*}
Theorems \ref{MainThm1} and \ref{MainThm2} are still valid for polynomials $f(x,y)=g(x)+h(y)$ where $h$ and $g$ are polynomials with real coefficients having real critical points and satisfying that $\gcd(\deg(g),\deg(h))\leq 2$.
\end{conjecture*}

\end{document}